\documentclass[aip,reprint,footinbib,reqno]{amsart}

\usepackage{graphicx,graphics,caption, subcaption, enumerate,color}

\usepackage{amsmath,amssymb,amsfonts,bm,empheq}
\usepackage{cases}

\usepackage{comment}
\usepackage{soul}

\usepackage{epstopdf}

\usepackage[a4paper]{geometry} 
\usepackage{hyperref}
 \hypersetup{
    colorlinks=true,       
    linkcolor=red,          
     citecolor=green,        
     urlcolor=cyan           
 }

\usepackage{setspace}

\usepackage{diagbox} 
\usepackage{multirow}

\usepackage{lineno}

\newcommand{\norm}[1]{\left\Vert#1\right\Vert}
\newcommand{\inpd}[2]{\left\langle #1, #2 \right\rangle}

\newcommand{\set}[1]{\left\{#1\right\}}

\newcommand{\wt}[1]{\widetilde{#1}}

\newcommand{\Hess}{{H}}

\newcommand{\argmin}{ \operatornamewithlimits{argmin} }

\newtheorem{theorem}{Theorem}

\newtheorem{remark}{Remark}

\def\Hess{\mathbf{H}}
\newcommand{\wh}[1]{\widehat{#1}}
\def\mcM{\mathcal{M}}
\def\mcU{\mathcal{U}}

\def\kpa2_2{\frac{\kappa^2}{2}}

\def\sigm2_2{\frac{\sigma^2}{2}}
\def\nab_dot{\nabla\cdot}

\def\Hess{\mathbf{H}}
\def\Pj{\mathbf{P}}

\def\I{\mathbf{I}}

\def\1/N{\frac{1}{N}}

\def\H_1_u{H_u^{-1}}

\def\f12{\frac{1}{2}}

\begin{document}
 
\begin{center}
{\bf \Large Projection Method for Saddle Points of   Energy Functional  in $H^{-1}$ Metric}

\

Shuting Gu\footnote{Corresponding author.  School of Mathematical Sciences, South China Normal University, Guangzhou 510631, PR China. Email: shutinggu@m.scnu.edu.cn }
, Ling Lin\footnote{ School of Mathematics, Sun Yat-sen University, Guangzhou 510275, China. Email: linling27@mail.sysu.edu.cn}
, Xiang Zhou\footnote{ School of Data Science and Department of Mathematics, City University of Hong Kong, Tat Chee Ave, Kowloon, Hong Kong SAR. Email: xiang.zhou@cityu.edu.hk}

\end{center}
 \date{\today}

\section*{Abstract}

  Saddle points play important roles as the transition states
 of  activated process in  gradient system driven by  
 energy functional. However, for the same energy functional,  the saddle points, as well as 
   other stationary points, are different  in different metrics such as the $L^2$ metric and the $H^{-1}$ metric. The saddle point calculation in $H^{-1}$ metric is more challenging with much higher computational cost since 
   it involves higher order derivative in space and the inner product calculation needs to solve another Possion equation to get the $\Delta^{-1}$ operator. 
In this paper, we introduce the projection idea  to the existing saddle point search methods, 
     gentlest ascent dynamics (GAD) and   iterative minimization formulation (IMF),  
     to overcome this numerical challenge due to $H^{-1}$ metric.        
 Our new method in the $L^2$ metric only by carefully incorporates  a simple linear projection step.  We show that our projection method
maintains the same convergence speed of the original GAD and IMF,
but the new algorithm is much faster than the direct method for $H^{-1}$ problem. 
 The  numerical results of saddle points in  the one dimensional Ginzburg-Landau free energy and the two dimensional Landau-Brazovskii free energy in $H^{-1}$ metric are presented to demonstrate the efficiency
 of this new method.

\

{{\bf Keywords}: saddle point, transition state, projection method, gentlest ascent dynamics }
\

{{\bf  Mathematics Subject Classification (2010)}  Primary 65K05, Secondary  82B05	  }


\section{Introduction}

Saddle points have important physical meaning and have been of broad interest in chemistry, physics, biology and material sciences.
 In computational chemistry \cite{Energylanscapes},  one of the most important objects on the potential energy surface
 is the transition state, a special type of the saddle point with index-1, which is defined as the critical point with only one unstable direction.
 Such transition states are the bottlenecks on the most probable transition paths
 between different local wells.
 In recent years, a large number of numerical methods have been proposed and developed to efficiently compute these saddle points. Generally speaking, there are two classes:  path-finding methods and surface-walking methods. The former 
  includes the string method \cite{Ren2013, String2002} and the nudged elastic band method \cite{NEB1998}. These methods are to search the so-called minimum energy path (MEP). The points along the MEP with locally maximum energy value are then the index-1 saddle points. The later methods include the eigenvector following method \cite{Crippen1971}, the dimer method \cite{Dimer1999}, the activation-relaxation techniques \cite{ART1998},
   the gentlest ascent dynamics(GAD) \cite{GAD2011} and the  iterative minimization formulation (IMF) \cite{IMF2014,IMA2015}.
  They evolve a single state on the potential energy surface along the unstable direction, for example, the min-mode direction.
   
  There are different fixed points on different potential energy surfaces.  Here we will address that  even for the same energy functional, different stationary points (metastable states)  and saddle points can be obtained in different metrics such as the $L^2$ metric and the $H^{-1}$ metric. We take the Ginzburg-Landau free energy on a bounded  domain $\Omega$ for example 
  \begin{equation}
 \label{eqn:F_GL0}
F(\phi) = \int_\Omega \Big[ \frac{\kappa^2}{2}
|\nabla \phi(x)|^2+ f(\phi) \Big]\,dx, \quad f(\phi) = (\phi^2-1)^2/4,
\end{equation}
and   the following two gradient flows are commonly used in
physics models, depending on which metric is used for the gradient.
\begin{enumerate}
    \item 
 In $L^2$ metric: the (non-conserved) Allen-Cahn  (AC) equation \cite{AC-EQ}
\begin{equation}\label{AC_eq}
\frac{\partial\phi}{\partial t} = -\frac{\delta F}{\delta\phi}(\phi)
=\kappa^2 \Delta \phi - (\phi^3-\phi);
\end{equation}
and 
\item In $H^{-1}$ metric: the (conserved) Cahn-Hilliard (CH) equation \cite{CH-EQ}
\begin{equation}\label{CH_eq}
\frac{\partial\phi}{\partial t} = \Delta\frac{\delta F}{\delta\phi}=-\kappa^2 \Delta^2 \phi + \Delta (\phi^3-\phi).
\end{equation}
\end{enumerate} 
Here $\frac{\delta F}{\delta \phi}$ is the first order variation of $F$
in the $L^2$ sense. Nowadays, the Allen-Cahn and Cahn-Hilliard equations have been widely used in many complicated moving interface problems in materials science and fluid dynamics through a phase-field approach, for instance, \cite{ShenYang, Fife, Bates, Brachet}.

The inner product and the norm in $H^{-1}$
metric can be rewritten
in terms of  the $L^2$ product as follows:
\begin{equation}\label{-1to2}
 \norm{\phi}_{H^{-1}}^2 = \big\langle (-\Delta)^{-1}\phi,\phi \big\rangle_{L^2}, ~~
 \langle\phi,\psi\rangle_{H^{-1}} = \big\langle(-\Delta)^{-1}\phi,\psi \big\rangle_{L^2},
\end{equation}
 where    $(-\Delta)^{-1}$,  a bounded positive self-adjoint linear operator,
 is the inverse of $-\Delta $ subject to certain boundary condition\cite{Dawson1989}. 
 The dynamics \eqref{AC_eq} and \eqref{CH_eq} are the gradient flows of the same energy functional \eqref{eqn:F_GL0} in $L^2$ metric and $H^{-1}$ metric, respectively. 
 It is clear that these two gradient flows have distinctive dynamics and properties.
The Cahn-Hilliard equation \eqref{CH_eq} preserves the mass $\int_\Omega \phi \, dx$ while the Allen-Cahn  does not.
 We are intertested in the stationary states of these two dynamics.
With the same boundary condition, the stationary states of dynamics \eqref{AC_eq} (with the sufficient regularity such as in the Sobolev $H^4(\Omega)$ space) are  the stationary states of the dynamics \eqref{CH_eq}, but not vice versa.
 
 It  takes  more computational cost to calculate  the stationary points in CH equation than that in AC equation, because
 the dynamics in the $H^{-1}$ metric \eqref{CH_eq} is a fourth order derivative equation in space,  two order higher than that in the $L^2$ metric \eqref{AC_eq}. 
 What is worse is that any computation  involving the inner product calculation in $H^{-1}$ metric needs to calculate the $\Delta^{-1}$ operator (see \eqref{-1to2}) by solving a Poisson equation. So if one only wants to locate the fixed points(stationary points or saddle points) instead of capturing the time evolution in $H^{-1}$ metric, it is much less efficient to use dynamics in $H^{-1}$ metric such as the CH equation.
  
Since the main difference between the dynamics in $L^2$ metric and $H^{-1}$ metric is whether the mass is conservation, our idea to handle the above challenges is to add mass conservation constrains into the $L^2$ metric dynamics. 
This conservation can be enforced by a projection operator. 
In the work of \cite{LLinProj2010},   the projected Allen-Cahn equation  
 \begin{equation}\label{ProAC}
 \frac{\partial u}{\partial t} = \Pj (-\frac{\delta F}{\delta u} )
 \end{equation}
was proposed as  a counterpart of the Cahn-Hillard equation to search different phases in diblock copolymers.  $\Pj$ in \eqref{ProAC} is the orthogonal projection operator onto the confined subspace satisfying the mass conservation.
 The  Cahn-Hilliard equation    \eqref{CH_eq} and the projected Allen-Cahn equation \eqref{ProAC}
then both preserve the mass, although the  gradient-descent trajectories and the transition paths are different\cite{Tiejun_1D}. 
One important fact is that \eqref{ProAC} and  \eqref{CH_eq} share   the same stationary points (metastable states) and the saddle points   if they have the same mass.
\cite{Tiejun_1D}  further compared the stochastic models 
arising from these two dynamics (\eqref{CH_eq} and \eqref{ProAC})  for the noise-induced transitions.
They showed the subtle difference in transition rates and minimum energy paths in the two stochastic models. 
For our purpose of locating the saddle point in this article,  we utilize the equivalence of saddle points of  \eqref{CH_eq} and \eqref{ProAC}  and solve 
   the saddle points of the Cahn-Hilliard equation \eqref{CH_eq}    
   by solving the projected Allen-Cahn equation \eqref{ProAC} (in $L^2$ metric). 
 
 Compared with the stationary points, people are more concerned about the saddle points for rare event study. 
  In \cite{convex_IMF}, the IMF has been  applied to locate the saddle point of an energy functional in $H^{-1}$ metric directly. However, as mentioned before, this directly is quite expensive in computation.   
 Considering the equivalence of the fixed points for CH equation and projected AC equation, we propose  to locate the saddle points  of the AC equation with the mass conservation constrain. 
 Recently, several methods have been developed to locate the saddle point with constrains.  
 \cite{LZ2009}    developed a constrained string method for finding the saddle points subject to constraints.   \cite{DuJCP2012} studied the constrained shrinking dimer dynamics  to locate saddle points associated with an energy functional defined on a constrained manifold. 
\cite{MAM_XZ} considered noise-induced transition paths in randomly perturbed dynamical systems on a smooth manifold.
 Besides,  the papers of  iterative minimization formulation (IMF)   \cite{IMF2014,IMA2015} have included the discussions on the projection idea for   saddle point on manifold. 
But in these works, the constraints are externally   imposed  and thus require higher computational cost than the unconstrained problems.
Our motivation  here is totally different. The question we considered here is essentially an unconstrained problem since the mass is conserved automatically in $H^{-1}$ metric.  We transform a difficult 
unconstrained problem into a less challenging constraint problem
and we only work on the orthogonal projection for mass conservation.  This method can reduce the computational cost efficiently since it can not only avoid a higher order equation solving but also escape from the $\Delta^{-1}$ operator calculation. Furthermore, we verify that the projected IMF can ensure the same convergence rate as the  original IMF. 
Finally, we remind the readers that if one really looks for the noise-induced transition paths in $H^{-1}$ sense, then the true dynamics like the CH equation \eqref{CH_eq} is still necessary, although our method for saddle points can assist this path-finding task; see details in \cite{Tiejun_1D}.

   The paper is organized as follows. Section \ref{Review} is a short review of two main methods for saddle points: the IMF and the GAD. In section \ref{Pro_GAD_IMF}, we first present the application of the IMF in the $H^{-1}$ metric,  and then
    propose the mathematical formulation of the projected IMF and the convergence result of the projected IMF. The projected GAD is also presented here.  In section \ref{Num_ex}, in order to validate the efficiency of our new   method, we test two numerical examples: the saddle points of the one dimensional Ginzburg-Landau free energy and the two dimensional Landau-Bravoskii free energy in $H^{-1}$ metric. Finally we make the conclusion.

\section{Review}\label{Review}

In this section, we will review  two main methods for saddle points: the  IMF and the GAD, from which the projected IMF and the projected GAD in the next Section will be proposed. 

\subsection{Iterative minimization formulation(IMF)}\label{Review_IMF}
 We first review the iteration minimization formulation (IMF) in \cite{IMF2014}.  Suppose $\mcM $ is  a function space equipped with the norm $\norm{\cdot}$
  and the inner product $\langle \cdot,\cdot \rangle$. The IMF to locate the saddle point of an energy functional $F(\phi)$ is the following iteration
\begin{numcases}{}
v^{(k+1)} = \argmin_{\| v \|=1} \inpd{\phi} {\Hess(\phi^{(k)})\phi},\label{IMF_direction}\\
\phi^{(k+1)} = \argmin_\phi L(\phi;\phi^{(k)},v^{(k+1)}),\label{IMF_position}
\end{numcases}
where $   \Hess = ~ \delta_\phi^2 F  $
is the second order variational operator of $F$, and
\begin{equation}\label{L_orig}
    \begin{split}
    L(\phi;\phi^{(k)}, v^{(k+1)}) = ~&(1-\alpha) F(\phi) + \alpha F\left( \phi- \inpd{v^{(k+1)}} { \phi-\phi^{(k)} }  v^{(k+1)} \right) \\
    ~ & - \beta F\left(\phi^{(k)} + \inpd{v^{(k+1)}} { \phi-\phi^{(k)} }  v^{(k+1)} \right).
    \end{split}
 \end{equation}
 $\alpha$ and $\beta$ are two parameters, and $\alpha+\beta>1$. Two special choices for $\alpha$ and $\beta$ are:
(i) $(\alpha, \beta) = (2,0),$ then $L(\phi;\phi^{(k)},v) = -F(\phi) + 2 F(\phi - \inpd {v} {\phi-\phi^{(k)} } v )$;
 (ii) $(\alpha, \beta) = (0,2),$ then $L(\phi;\phi^{(k)},v) = F(\phi) - 2 F(\phi^{(k)} + \inpd {v} {\phi-\phi^{(k)} } v )$.
  \eqref{IMF_direction} is called the ``rotation step" and \eqref{IMF_position} is the ``translation step".  The main properties of the auxiliary objective functional $L(\phi;\phi^{(k)},v)$ when $\alpha + \beta >1$ are listed
 here for reference.
 

 \begin{theorem}[\cite{IMF2014}]
 \label{Th_IMF}~Suppose that $\phi^*$ is a (non-degenerate) index-1 saddle point of the functional $F(\phi)$, 
and the auxiliary functional $L$ is defined by $(\ref{L_orig})$ with $\alpha + \beta >1$, then\\
 $(1)$ a neighbourhood  $\mcU$ of $\phi^*$ exists  such that for any $\phi \in \mcU$, $L(\phi;\phi^{(k)},v)$ is strictly convex in $\phi\in\mcU$ and thus has a unique minimum in $\mcU$;\\
 $(2)$ define the mapping $\Phi: \phi\in\mcU \rightarrow \Phi(\phi)\in\mcU$ to be the unique minimizer of $L$ in $\mcU$ for any $\phi\in\mcU$. Further assume that $\mcU$ contains no other stationary points of $F$ except for $\phi^*$. Then the mapping $\Phi$ has only one fixed point $\phi^*$;\\
 $(3)$ the mapping $\phi \to \Phi(\phi)$ has a quadratic convergence rate.
 \end{theorem}
 

\subsection{Gentlest ascent dynamics(GAD)}\label{GAD}
The GAD for a gradient system $\dot{\phi} = -\delta_\phi F(\phi)$ is
\begin{subnumcases}{\label{GAD-g}}
  \dot{\phi} = -\delta_\phi F(\phi) + 2 \frac{\inpd{\delta_\phi F(\phi)}{ v}}{\inpd{v}{v}} v,\label{GAD-g-x}\\
\gamma \dot{v}(t)  =   - \delta_\phi^2 F(\phi)v + \inpd{v} {\delta_\phi^2 F(\phi)v}v, \label{GAD-g-v}
\end{subnumcases}
 where $\delta_\phi^2 F(\phi)$ is the second order variational derivative of the energy functional $F(\phi)$.  $\gamma > 0 $ is the relaxation parameter. A large $\gamma$ means
 a fast dynamics for the direction variable
 $v(t)$ towards the steady state.
 For a frozen $\phi$, this steady state is the min mode of $\delta_\phi^2 F(\phi)$:
  the eigenvector corresponds to the smallest eigenvalue of $\delta_\phi^2 F(\phi)$.
 
  \begin{theorem}[\cite{GAD2011}]

The (linearly) stable critical point of the GAD \eqref{GAD-g} corresponds to the index-1 saddle point of the original dynamics $\dot{\phi} = -\delta_\phi F(\phi)$, i.e.,

$(1)$  If $(\phi^*,v^*)$ is a stable critical point of the GAD, then $\phi^*$ is a saddle point of $F(\phi)$; 

 $(2)$  If $\phi^*$ is an index-1 saddle point of $F(\phi)$
 with the eigenvector $v^*$, then $(\phi^*,v^*)$ is a
  stable critical point of the GAD.
 
 \end{theorem}
 
 \begin{remark}
 In the IMF, there are two levels of iterations: the rotation step and the translation step.  In general, it requires many iteration steps to get $\phi^{(k+1)}$ for the translation step, but it is not necessary to do so in practice. If the two subproblems of the IMF moves forward only one iteration step,  the IMF becomes exactly the GAD in continuous time limit minimizations.

\end{remark}

\section{Main methods}\label{Pro_GAD_IMF}

We present the main methods of projection here
by starting with the formulation  
in the $H^{-1}$ space where the inner product  $\inpd{\cdot}{\cdot}$ becomes $\inpd{\cdot}{\cdot}_{H^{-1}}$.

\subsection{The IMF in $H^{-1}$ metric}\label{IMF_H-1}
Formally, the IMF in the spatially extended system to locate the saddle point of $F(\phi), \phi\in C(\Omega)$ in $H^{-1}$ metric is:
\begin{numcases}{}
v^{(k+1)} = \operatornamewithlimits{argmin}_{\|v\|_{H^{-1}}=1} \left\langle v, \widetilde\Hess(\phi^{(k)})v \right\rangle_{H^{-1}},\label{IMF_oo1}\\
\phi^{(k+1)} = \operatornamewithlimits{argmin} L(\phi;\phi^{(k)},v^{(k+1)}),\label{IMF_oo2}
\end{numcases}
where 
\begin{equation}\label{Hess_H1}
\wt\Hess=\delta_\phi^2 F(\phi)|_{H^{-1}} = -\Delta \delta_\phi^2 F(\phi) = -\Delta\Hess, 
\end{equation}
and
\begin{equation}\label{L_H_1}
    \begin{split}
    L(\phi;\phi^{(k)}, v^{(k+1)}) = ~&(1-\alpha) F(\phi) + \alpha F\left( \phi- \inpd{v^{(k+1)}} { \phi-\phi^{(k)} }_{H^{-1}}  v^{(k+1)} \right) \\
    ~ & - \beta F\left(\phi^{(k)} + \inpd{v^{(k+1)}} { \phi-\phi^{(k)} }_{H^{-1}}  v^{(k+1)} \right).
    \end{split}
 \end{equation}
Recall  $\Hess = \delta_\phi^2 F(\phi)$ is the second order variational operator of $F$ w.r.t. $\phi$ in $L^2$ metric. 
 For convenience, in this paper, we take 
$\alpha=0$, $\beta=2$, then 
$$L(\phi) = F(\phi)- 2 F(\hat\phi),$$ 
with 
\begin{equation}\label{hatphi_H1}
\hat\phi = \phi^{(k)} + \inpd{v^{(k+1)} }{(\phi-\phi^{(k)})}_{H^{-1}} v^{(k+1)},
\end{equation}
 where the inner product in $H^{-1}$ metric is defined by the $L^2$ product: $ \langle u,v\rangle_{H^{-1}} = \big\langle(-\Delta)^{-1}u,v \big\rangle_{L^2}$.

In $H^{-1}$ metric, $\phi$ is mass conserved, $\int_\Omega \phi(x) \, dx = m$. So any eigenvectors of $\wt\Hess$ satisfies $\int_\Omega \psi(x) \, dx = 0$. Thus the eigenvalue problem \eqref{IMF_oo1} can be rewritten as
\begin{equation}\label{eig-CH}
\begin{cases}
\wt{\Hess}(\phi)\psi  = \lambda \psi,  \\
   \int_\Omega \psi(x) \,dx = 0,
\end{cases}
\end{equation}
subject to some boundary condition. 
Define the Rayleigh quotient 
\[
\wt{\mathcal{R}} (\psi)= \frac{\inpd{\psi}{ \wt{\Hess}\psi}_{H^{-1}} }{\norm{\psi}^2_{H^{-1}}},
\]
and thus the min-mode is the minimizer of the problem
\begin{equation}\label{R-CH}
  \argmin_{\psi} \set{\wt{\mathcal{R}} (\psi):  {\int_\Omega \psi \, dx =0, ~~\norm{\psi}_{H^{-1}}=1} }.
\end{equation}
After the min-mode is obtained, the subproblem of minimizing the auxiliary functional \eqref{IMF_oo2} is then solved by evolving the gradient flow
\begin{equation}\label{H_1metric}
  \frac{\partial\phi}{\partial t} = \Delta\frac{\delta L}{\delta\phi}(\phi)
= \Delta \left(\frac{\delta F}{\delta\phi}(\phi) \right) + 2    \inpd{  \frac{\delta F}{\delta\hat\phi}(\hat{\phi} ) }{v}_{L^2} v,
 \end{equation}
 where $\hat\phi$ is defined in \eqref{hatphi_H1}.
By solving \eqref{R-CH} and \eqref{H_1metric}, one can get the saddle point of $F(\phi)$ in $H^{-1}$ metric.  The readers can refer to \cite{convex_IMF} for details.

For the IMF in the $H^{-1}$ metric, we can see in \eqref{hatphi_H1} that the $H^{-1}$ inner product calculation requires to get the $-\Delta^{-1} $ operator first. This can be transformed to a Poisson equation $\Delta w = -u$ and it takes large computational cost. This is why we consider the projected method in this note to locate the saddle point in $H^{-1}$ metric. In the next two subsections, we first present the projected IMF and then propose the projected GAD.

\subsection{The Projected IMF}
\label{Projected_IMF}
In this part, we propose the projected iterative minimization formulation to locate the saddle point of an energy functional $F(\phi)$ in the $H^{-1}$ metric.  Since the mass is preserved in $H^{-1}$ metric, we introduce the projection $\Pj$
\begin{equation}\label{Pj}
\Pj u:= u - \frac{1}{|\Omega|}\int_\Omega u(x)\, dx
\end{equation}
 onto the linear subspace
$H_0=\set{u \in L^2 : \int_\Omega u(x) dx =0}$.  One can show that $\Pj$ has the following properties:
\begin{enumerate}
\item $\Pj^2 = \Pj$;
\item $\Pj u \in H_0, \forall u \in L^2$;
\item $\Pj v = v, \forall v \in H_0$;
\item  $ \langle v, \Pj w \rangle_{L^2} = \langle \Pj v, w \rangle_{L^2},  \forall v \in L^2 $ and $\forall w \in H_0 $ . 
\end{enumerate}
In fact,
\begin{align*}
\Pj^2 u &= \Pj ( u - \frac{1}{|\Omega|}\int_\Omega u(x)\, dx) \\
       & = u - \frac{1}{|\Omega|}\int_\Omega u(x)\, dx - \frac{1}{|\Omega|}\int_\Omega (u(x) -\frac{1}{|\Omega|}\int_\Omega u(x)\, dx) \, dx \\
    &   = \Pj u, \quad \forall u.
\end{align*}
Besides,  one can easily show $\Pj u \in H_0$, since
\begin{equation}\label{int_Pu}
\int_\Omega \Pj u \, dx = \int_\Omega (u - \frac{1}{|\Omega|} \int_\Omega u(x) \, dx ) \, dx =0, \quad \forall u.
\end{equation}
The third property is  obvious. For the last one, when $ w \in H_0, v \in L^2$, we have
\begin{align*}
 \langle v, \Pj w \rangle_{L^2}  - \langle \Pj v, w \rangle_{L^2} = & \langle v, w \rangle_{L^2}  - \langle \Pj v, w \rangle_{L^2}  =   \langle v-\Pj v, w \rangle_{L^2}.
\end{align*}
Since $\Pj v \in H_0, v\in L^2$ and $\Pj $ is the projection from $L^2$ to $H_0$, 
we obtain $$ v - \Pj v \in H_0^{\perp}  \quad \Rightarrow   \quad  \langle v-\Pj v, w \rangle_{L^2} = 0, $$
that is,
$$   \langle v, \Pj w \rangle_{L^2}  = \langle \Pj v, w \rangle_{L^2} . $$

For the rotation step \eqref{IMF_oo1} in the IMF,  the following equivalence can be get
\begin{align*}
\left\langle v, \widetilde \Hess v \right\rangle_{H^{-1}} 
=  \left\langle v, -\Delta \Hess v \right\rangle_{H^{-1}} 
 =  \left\langle v, -(-\Delta)^{-1} \Delta \Hess  v \right\rangle_{L^2} 
 =  \left\langle v, \Hess v \right\rangle_{L^2}.
\end{align*}
When $v\in H_0$, by the last two properties of $\Pj$, we have 
$$ \left\langle v, \Hess v \right\rangle_{L^2} = \left\langle \Pj v,  \Hess \Pj v \right\rangle_{L^2} = \left\langle v, \Pj \Hess \Pj v \right\rangle_{L^2}, $$
so the eigenvector problem \eqref{IMF_oo1} can be equivalent transformed to
$$ v^{(k+1)} = \operatornamewithlimits{argmin} \left\langle v, \Pj \Hess  \Pj v \right\rangle_{L^2} $$
without regarding to the ``length" (norm) of $v$ since only the direction of $v$ matters in the translation step.

 Besides, due to the equivalence of saddle points of $F(\phi)$ in $H^{-1}$ metric (Cahn-Hilliard equation) and in $L^2$ metric with projection (projected Allen-Cahn equation),  all the terms in equation \eqref{IMF_oo1} and \eqref{IMF_oo2} including  the (first-order and second-order) variation, the inner-product and the norm in $H^{-1}$ metric can be changed to the $L^2$ metric onto the confined subspace $H_0$. It can be verified that the relation of the variations in $L^2$ space  and its subspace $H_0$ is
$$ \mu_1 = \Pj \mu_2,  \quad  \widehat\Hess = \Pj \Hess \Pj, $$
where $\mu_1$ and $\mu_2 $ are the first-order variations of $F(\phi)$ in $H_0$ and $L^2$, respectively.  $\widehat\Hess $ and $\Hess$ are the second-order variations of $F(\phi)$ in $H_0$ and $L^2$, respectively.

So the projected IMF written in terms of $\inpd{ \cdot }{\cdot }_{L^2}$  is 
 \begin{numcases}{}
v^{(k+1)} = \operatornamewithlimits{argmin}_{\|v\|_{L^2}=1} \left\langle v, \wh\Hess(\phi^{(k)})v \right\rangle_{L^2},\label{IMF_pro1}\\
\phi^{(k+1)} = \operatornamewithlimits{argmin}_{\int_\Omega \phi(x) dx=m} L(\phi;\phi^{(k)},v^{(k+1)}),\label{IMF_pro2}
\end{numcases}
where $\wh\Hess=\Pj \Hess \Pj, \Hess = \delta^2_\phi F(\phi)$.  $L(\phi)
=F(\phi)- 2 F(\hat\phi)$, with 
\begin{equation}\label{phi_hat}
\hat\phi = \phi^{(k)} + \inpd{v^{(k+1)} }{(\phi-\phi^{(k)})}_{L^2} v^{(k+1)}.
\end{equation}
The eigenvector problem \eqref{IMF_pro1} is equivalent to 
 \begin{equation}\label{eig-Proj-AC}
\begin{cases}
\Pj \Hess \Pj \psi  = \lambda \psi,  \\
   \int_\Omega \psi(x) \,dx = 0,
\end{cases}
\end{equation}
subject to some boundary condition. In this paper, we consider the periodic boundary condition only. 
The Rayleigh quotient in this case is 
\[
\wh{\mathcal{R}} (\psi)= \frac{\inpd{\psi}{ \Pj \Hess \Pj \psi }_{L^2} }{\norm{\psi}^2_{L^2}},
\]
and thus the min-mode of $\wh\Hess$ is the minimizer of the problem
\begin{equation}\label{R-Proj-AC}
  \argmin_{\psi} \set{\wh{\mathcal{R}} (\psi):  {\int_\Omega \psi \, dx =0, ~~\norm{\psi}_{L^2}=1} }.
\end{equation}
After the min-mode is obtained, the subproblem of minimizing the auxiliary functional \eqref{IMF_pro2} is then solved by evolving the gradient flow
\begin{equation}\label{dynamics_pro}
\frac{\partial\phi}{\partial t} = - \Pj \frac{\delta L}{\delta \phi}(\phi)
= -\Pj \frac{\delta F}{\delta \phi}(\phi)
+2  \inpd{v}{ \frac{\delta F}{\delta \hat\phi}(\hat\phi)}_{L^2} \Pj v,
\end{equation}
where
 $$\hat\phi = \phi^{(k)} + \inpd{v^{(k+1)} }{(\phi-\phi^{(k)})}_{L^2} v^{(k+1)}. $$  
We can show that \eqref{dynamics_pro} ensures mass conservation automatically
\[
\frac{\partial}{\partial t} \int_\Omega  \phi(x) \, dx = -\int_\Omega \Pj \frac{\delta L}{\delta\phi} \, dx = 0,
\]
thanks to \eqref{int_Pu}.
Thus, by solving \eqref{R-Proj-AC} and \eqref{dynamics_pro}, we can get the saddle point of $F(\phi)$ in $H^{-1}$ metric.  It is easy to find that \eqref{dynamics_pro} is two order lower in spatial derivative than \eqref{H_1metric}, which is the gradient flow in $H^{-1}$ metric directly. Besides, the inner product in \eqref{dynamics_pro} is in $L^2$ metric which avoids the  $\Delta^{-1}$ operator calculation. We can also apply the convex splitting method to \eqref{dynamics_pro} to construct a large time step size scheme as in \cite{convex_IMF}.

Denote this mapping for the iteration as $\Phi(\phi)$, we shall show that the Jacobian matrix of $\Phi(\phi)$ in the projection sense vanishes at the index-1 saddle point. This implies that the projected IMF is of quadratic convergence rate.

\subsection{Convergence Results}

 \begin{theorem}\label{Th_2nd_rate}
~Suppose that $\phi^*$ is a (non-degenerate) index-1 saddle point of the functional $F(\phi)$, which satisfies that the second order variational derivative $\delta^2_\phi F(\phi)$ is continuous. For each $\phi$, $v(\phi)$ is the normalized  eigenvector corresponding to the smallest eigenvalue of the matrix $\wh\Hess = \Pj \Hess \Pj, \Hess = \delta^2_\phi F $, i.e., 
 $$ v(\phi) = \argmin_{\|u\|=1} u^T  \wh\Hess(\phi)  u.$$ Take $\alpha = 0, \beta = 2,$ and the auxiliary functional is $L(\phi) = F(\phi) - 2F(\hat\phi)$, then\\
 $(1)$ $\phi^*$ is local minimizer of $L(\phi; \phi^*,v)$;\\
 $(2)$ a neighbourhood  $\mcU$ of $\phi^*$ exists  such that for any $\phi \in \mcU$, $L(\phi;\phi^{(k)},v)$ is strictly convex in $\phi\in\mcU$ and thus has a unique minimum in $\mcU$;\\
 $(3)$ define the mapping $\Phi: \phi\in\mcU \rightarrow \Phi(\phi)\in\mcU$ to be the unique minimizer of $L$ in $\mcU$ for any $\phi\in\mcU$. Further assume that $\mcU$ contains no other stationary points of $F$ except for $\phi^*$. Then the mapping $\Phi$ has only one fixed point $\phi^*$;\\
 $(4)$ 
 $\Phi(\phi)$ is differentiable in $\mcU$ and $\Pj \Phi^\prime(\phi^*) \Pj = 0$. Thus the mapping $\phi \to \Phi(\phi)$ has a local quadratic convergence rate. 
 \end{theorem}
 \begin{proof}
 The proof of the first three conclusions can be generalized from the finite space to the infinite space based on the proof of Theorem 3.1 in \cite{IMF2014} without difficuty. The main difference is the proof of the quadratic convergence rate. Here, we only give the details of the final conclusion.

In fact, the first order variational derivative of $L(\phi; \phi^{(k)},v(\phi))$ can be calculated as
$$ \delta_\phi L(\phi; \phi^{(k)},v(\phi)) = \delta_\phi F(\phi) - 2 \inpd{v}{\delta_{\hat\phi}F(\hat\phi)}_{L^2} v. $$ 
At each $\phi^{(k)} \in \mcU$, the mapping $\Phi(\phi^{(k)})$ satisfies the first order equation 
$\Pj \delta_\phi L(\Phi(\phi^{(k)}),\phi^{(k)},v(\phi^{(k)})) = 0$, that is, 
\begin{equation}\label{PgradL_0}
\Pj \delta_\phi F(\Phi(\phi^{(k)})) - 2 \inpd{v(\phi^{(k)})}{\delta_{\hat\phi}F(\hat\phi)}_{L^2} \Pj v(\phi^{(k)}) = 0.
\end{equation}
Take derivative w.r.t. $\phi^{(k)}$ on both sides of \eqref{PgradL_0}, we get
\begin{align}\label{grad_PgradL_0}
& \Pj^2 \Hess(\Phi(\phi^{(k)})) \Pj^2 \Phi^\prime(\phi^{(k)}) \Pj - 2 \inpd{v(\phi^{(k)})}{\delta_{\hat\phi}F(\hat\phi)}_{L^2} \Pj^2 J(\phi^{(k)}) \Pj \nonumber\\
& -  2\inpd{\Pj J(\phi^{(k)}) \Pj }{\delta_{\hat\phi}F(\hat\phi)}_{L^2} \Pj v(\phi^{(k)}) \nonumber\\
& - 2\inpd{v(\phi^{(k)})}{\Pj \Hess(\hat\phi)\Pj \hat\phi^\prime }_{L^2} \Pj v(\phi^{(k)}) = 0,
\end{align}
where $\Hess = \delta_\phi^2 F, J(\phi^{(k)}) = \frac{\partial v(\phi^{(k)})}{\partial\phi^{(k)}}$ and
\begin{align*}
\hat\phi^\prime(\phi^{(k)}) = \Pj & + \inpd{v(\phi^{(k)})}{\Phi(\phi^{(k)}) - \phi^{(k)}}_{L^2} \Pj J(\phi^{(k)}) \Pj \\
& + \inpd{\Pj J(\phi^{(k)})\Pj}{\Phi(\phi^{(k)})-\phi^{(k)}}_{L^2} v(\phi^{(k)}) \\
& + \inpd{v(\phi^{(k)})}{\Pj \Phi^\prime(\phi^{(k)}) \Pj - \Pj }_{L^2} v(\phi^{(k)}).
\end{align*}
Let $\phi^{(k)} = \phi^*$ be the saddle point, we have $\Phi(\phi^*) = \phi^*, \hat\phi = \phi^*, \delta_{\hat\phi}F(\phi^*) = 0$ and $\hat\phi^\prime = \Pj + \inpd{v(\phi^{(k)})}{\Pj \Phi^\prime(\phi^*) \Pj - \Pj }_{L^2} v(\phi^*),$ thus \eqref{grad_PgradL_0} becomes
\begin{align*}
& \Pj \Hess(\phi^*) \Pj^2 \Phi^\prime(\phi^*) \Pj \nonumber \\
 = & ~ 2\inpd{v(\phi^*)}{\Pj \Hess(\phi^*)\Pj  [\Pj + \inpd{v(\phi^*)}{\Pj \Phi^\prime(\phi^*) \Pj - \Pj }_{L^2} v(\phi^*)] }_{L^2} \Pj v(\phi^*),
\end{align*}
which can be simplified as
\begin{align}\label{final_eq}
( \Pj \Hess(\phi^*) \Pj - 2\lambda vv^T) \Pj \Phi^\prime(\phi^*) \Pj = 0,
\end{align}
by denoting $u^Tv = \inpd{u}{v}$, applying $\Pj \Hess \Pj v = \lambda v$ and $\Pj \Hess \Pj ( \I - vv^T) \Pj = 0$.
\eqref{final_eq} implies that
$$  \Pj \Phi^\prime(\phi^*) \Pj = 0. $$
One can carry out the second order derivative of $\Phi(\phi)$ at $\phi^*$ further and observe that  $  \Pj \Phi^{\prime\prime}(\phi^*) \Pj = 0$ does not trivially hold. Thus the iteration $\phi \to \Phi(\phi) $ locally converges to $\phi^*$ with the quadratic rate.
 \end{proof}

\begin{remark}
  Theorem \ref{Th_2nd_rate} is also applicable for any auxiliary functional $L$ only if $\alpha+\beta >1$. Here we take $\alpha = 0, \beta =2 $ just for convenience. $\alpha$ and $\beta$ are  defined in Section \ref{Review_IMF}.
\end{remark}

\begin{remark}
We are dealing with a linear constraint here, so the projection $\Pj$ is the standard orthogonal
projection. For general nonlinear constraints giving rise to a submanifold, the projection should 
follow the  geodesic distance on the submanifold exactly to ensure the quadratic convergence rate in IMF \cite{IMA2015}.
\end{remark}
 
\subsection{Projected GAD}
We now present the projected GAD to calculate the saddle point of the energy functional $F(\phi)$ in $H^{-1}$ metric. 
For comparison, we put the original GAD \eqref{GAD-g}  in Section \ref{GAD} here:

 \begin{numcases}{}
 \frac{\partial \phi}{\partial t} = -\delta_\phi F(\phi) + 2 \frac{\inpd{\delta_\phi F(\phi)}{ v}_{L^2} }{\inpd{v}{v}_{L^2}} v,\label{GAD-ori-x}\\
\gamma \frac{\partial v}{\partial t}  =  - \delta_\phi^2 F(\phi)v + \inpd{v} {\delta^2_\phi F(\phi)v}_{L^2} v. \label{GAD-ori-v}
  \end{numcases}
  
By using the projection $\Pj$ in \eqref{Pj}, the projected GAD is given as follows:

 \begin{numcases}{}
 \frac{\partial \phi}{\partial t} = -\Pj\delta_\phi F(\phi) + 2 \frac{\inpd{\delta_\phi F(\phi)}{ v}_{L^2} }{\inpd{v}{v}_{L^2}} \Pj v,\label{GAD-Pj-x}\\
\gamma \frac{\partial v}{\partial t}  =  - \Pj \delta_\phi^2 F(\phi) \Pj v + \inpd{v} {\Pj \delta^2_\phi F(\phi) \Pj v}_{L^2} v. \label{GAD-Pj-v}
 \end{numcases}
By integrating w.r.t. $x$ on both sides of \eqref{GAD-Pj-v}, we get
 \begin{align*}
 \gamma \frac{\partial}{\partial t} \int_\Omega v \, dx & = - \int_\Omega \Pj \delta_\phi^2 F(\phi) \Pj v \, dx + \inpd{v} {\Pj \delta^2_\phi F(\phi) \Pj v}_{L^2} \int_\Omega v \, dx \\
 & = \inpd{v} {\Pj \delta^2_\phi F(\phi) \Pj v}_{L^2} \int_\Omega v \, dx, 
 \end{align*}
 this is an ordinary differential equation of $\int_\Omega v \, dx$. Considering the initial condition, $\int_\Omega v_0 \, dx = 0$, one can easily get 
 \[ \int_\Omega v(x) \, dx = 0, \quad \forall v, \]
 and thus
 \[ \Pj v = v - \int_\Omega v(x) \,dx = v, \quad \forall v. \]
  Furthermore, by using the last property of $\Pj$,
  $$ \langle v, \Pj w \rangle_{L^2} = \langle \Pj v, w \rangle_{L^2},  \forall v \in L^2, \forall w \in H_0  .$$ 
   the projected GAD can be rewritten as
%

 \begin{numcases}{}
\frac{\partial \phi}{\partial t} = -\Pj\delta_\phi F(\phi) + 2 \frac{\inpd{\delta_\phi F(\phi)}{ v}_{L^2} }{\inpd{v}{v}_{L^2}} \Pj v,\label{GAD-Pj1-x}\\
\gamma \frac{\partial v}{\partial t} =  - \Pj \delta_\phi^2 F(\phi) \Pj v + \inpd{v} { \delta^2_\phi F(\phi) v}_{L^2} v. \label{GAD-Pj1-v}
\end{numcases}
By solving the equation \eqref{GAD-Pj1-x} and \eqref{GAD-Pj1-v}, we can calculate the saddle points of $F(\phi)$ in $H^{-1}$ metric.

\section{Numerical example}\label{Num_ex}

In this section, we will illustrate the above projection method by locating the transition state of the one dimensional Ginzburg-Landau free energy and the two dimensional Landau-Brazovskii free energy in the $H^{-1}$ metric.

\subsection{1D example: Ginzburg-Landau free energy}
Consider the one dimensional Ginzburg-Landau free energy on $[0,1]$,
\begin{equation}
 \label{eqn:F_GL}
F(\phi) = \int_0^1 \Big[ \frac{\kappa^2}{2}(\frac{\partial \phi}{\partial x})^2 + f(\phi) \Big]\,dx,
\end{equation}
 where $\phi(x)$ is an order parameter and $\kappa>0$. $f(\phi) = (\phi^2-1)^2/4$.
The first and the second order variation of $F(\phi)$ can be calculated as
\begin{align*}
\frac{\delta F}{\delta\phi}(\phi) & = -\kappa^2 \Delta \phi + \phi^3-\phi,  \\
\frac{\delta^2 F}{\delta\phi^2}(\phi) & = -\kappa^2 \Delta + 3\phi^2-1:= \Hess.
\end{align*}
 So the projected IMF is
  \begin{numcases}{}
v^{(k+1)} = \operatornamewithlimits{argmin}_{\|v\|=1} \left\langle v, \Pj\Hess\Pj(\phi^{(k)})v \right\rangle_{L^2},\label{IMF_ex1}\\
\phi^{(k+1)} = \operatornamewithlimits{argmin}_{\int_\Omega \phi(x) dx=m} L(\phi;\phi^{(k)},v^{(k+1)}),
\label{IMF_ex2}
\end{numcases}
with $L(\phi) = F(\phi) - F(\hat\phi)$, $\hat\phi$ is defined in \eqref{phi_hat}.
The second minimization  sub-problem \eqref{IMF_ex2} is solved by evolving the gradient flow:
$$ \frac{\partial\phi}{\partial t}  =  - \Pj \delta_\phi L(\phi),
 $$
 where
$$  - \Pj \delta_\phi L(\phi) = \! -\Pj \big[\!\!-\kappa^2 \Delta \phi + (\phi^3-\phi) \big] \! + \! 2  \inpd{v}{ -\kappa^2 \Delta \hat\phi + (\hat\phi^3-\hat\phi) }_{L^2} \Pj v, $$
here $\phi = \phi^{(k+1)}$. And the projected GAD is


  \begin{numcases}{}
\frac{\partial \phi}{\partial t} = -\Pj (-\kappa^2 \Delta \phi + (\phi^3-\phi)) + 2 \frac{\inpd{v}{-\kappa^2 \Delta \phi + (\phi^3-\phi)}_{L^2} }{\inpd{v}{v}_{L^2}} \Pj v,\label{GAD-exPj1-x}\\
\gamma \frac{\partial v}{\partial t}  =  \! - \! \Pj (-\kappa^2 \Delta \! + \! (3\phi^2-1)) \Pj v + \inpd{v} { -\kappa^2 \Delta v \! + \! (3\phi^2-1)v  }_{L^2} v.
\end{numcases}
We apply the finite difference scheme to achieve the numerical example.
  For the projected IMF, we further construct the following convex splitting scheme \eqref{Proj_AC_convex_scheme}  to discrete \eqref{IMF_ex2} in time.
  
   \begin{equation}\label{Proj_AC_convex_scheme}
     \begin{split}
         \frac{\phi^{n+1}-\phi^n}{\Delta t}  =  & \Pj \left[ \kappa^2 \Delta\phi - 2\phi - 2 \inpd{v}{\phi}v \right]^{n+1} \\
         & \!+\! \Pj \left[ -\phi^3 + 3\phi + 2\inpd{v}{ -\kappa^2\Delta\hat\phi + {\hat\phi}^3 } v \right]^n.
     \end{split}
 \end{equation}
 
   In the numerical test, we take $\kappa=0.04$, the initial mass $m=0.6$, and the mesh grid is $\{x_i = i h , i=0, 1, 2,\ldots, N\}. ~h = 1/N.$ $N=100$, $\Delta t = 0.1$. We use the periodic boundary condition in this example. We find that the saddle point of $F(\phi)$ in the $H^{-1}$ metric calculated by the projected IMF or the projected GAD is exactly the same as the result in \cite{convex_IMF} which applies the IMF in the $H^{-1}$ metric directly, see Figure \ref{figA:ProjIMF_transition states}. Besides, the quadratic convergence rate can also be observed when using the projected IMF; see Figure \ref{figB:ProjIMF_convergence_rate} for the convergence result. 
  In order to illustrate the advantage of this method, we make comparison of the CPU time required for the same iteration number  
   between the projected IMF  in $L^2$ metric and the  original IMF in $H^{-1}$ metric.  Table \ref{1D_case} shows the results for various initial states $\phi_{01}, \phi_{02} $  and $\phi_{03}$. One can find that the projected IMF in $L^2$ metric can save almost half computational cost compared with the IMF in $H^{-1}$ metric, especially for the large inner iteration number.

\begin{figure}[htbp]
\begin{center}
\begin{subfigure}[b]{0.45\textwidth}
\includegraphics[width=\textwidth]{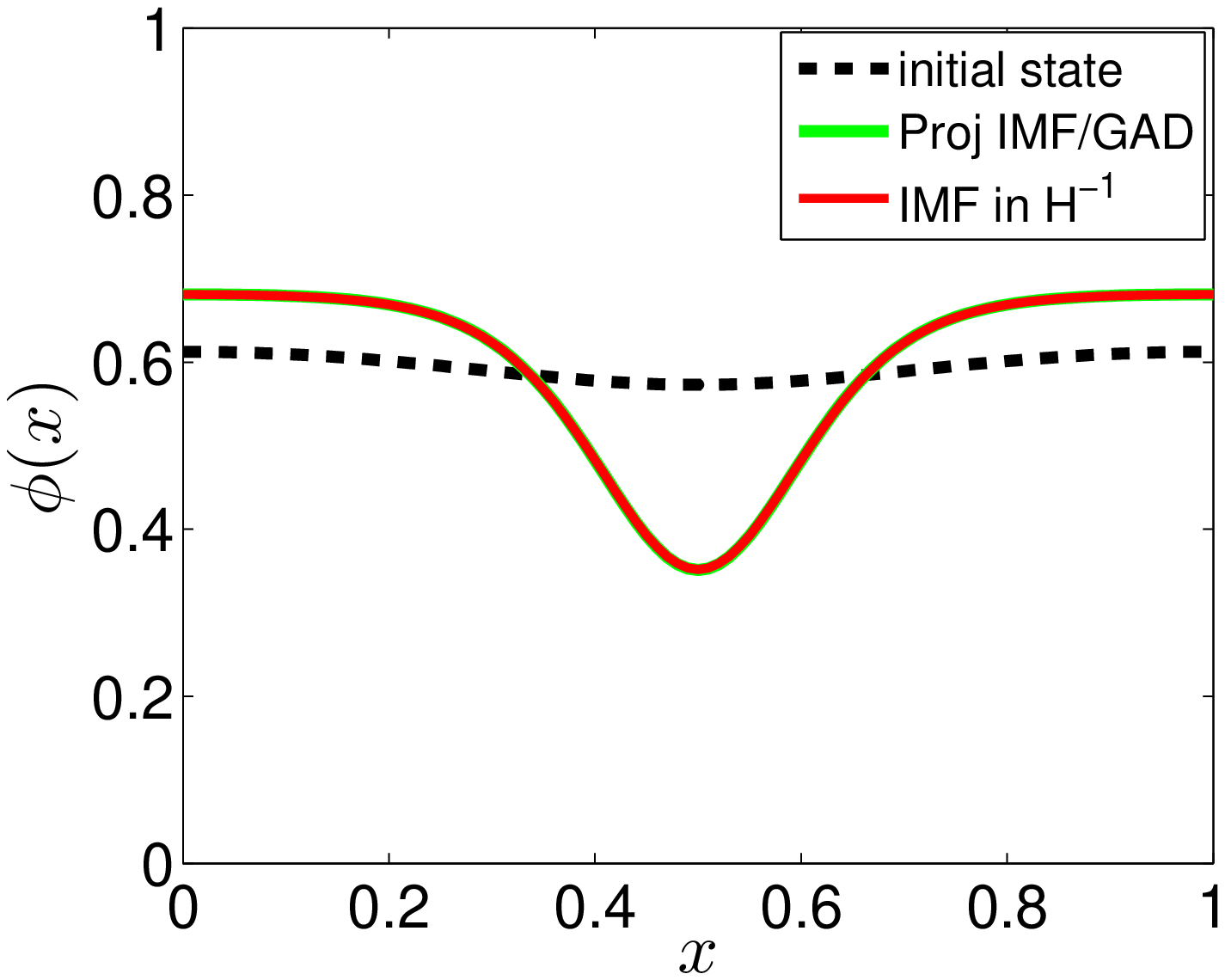}
\caption{transition states}
\label{figA:ProjIMF_transition states}
\end{subfigure}
\begin{subfigure}[b]{0.45\textwidth}
\includegraphics[width=\textwidth]{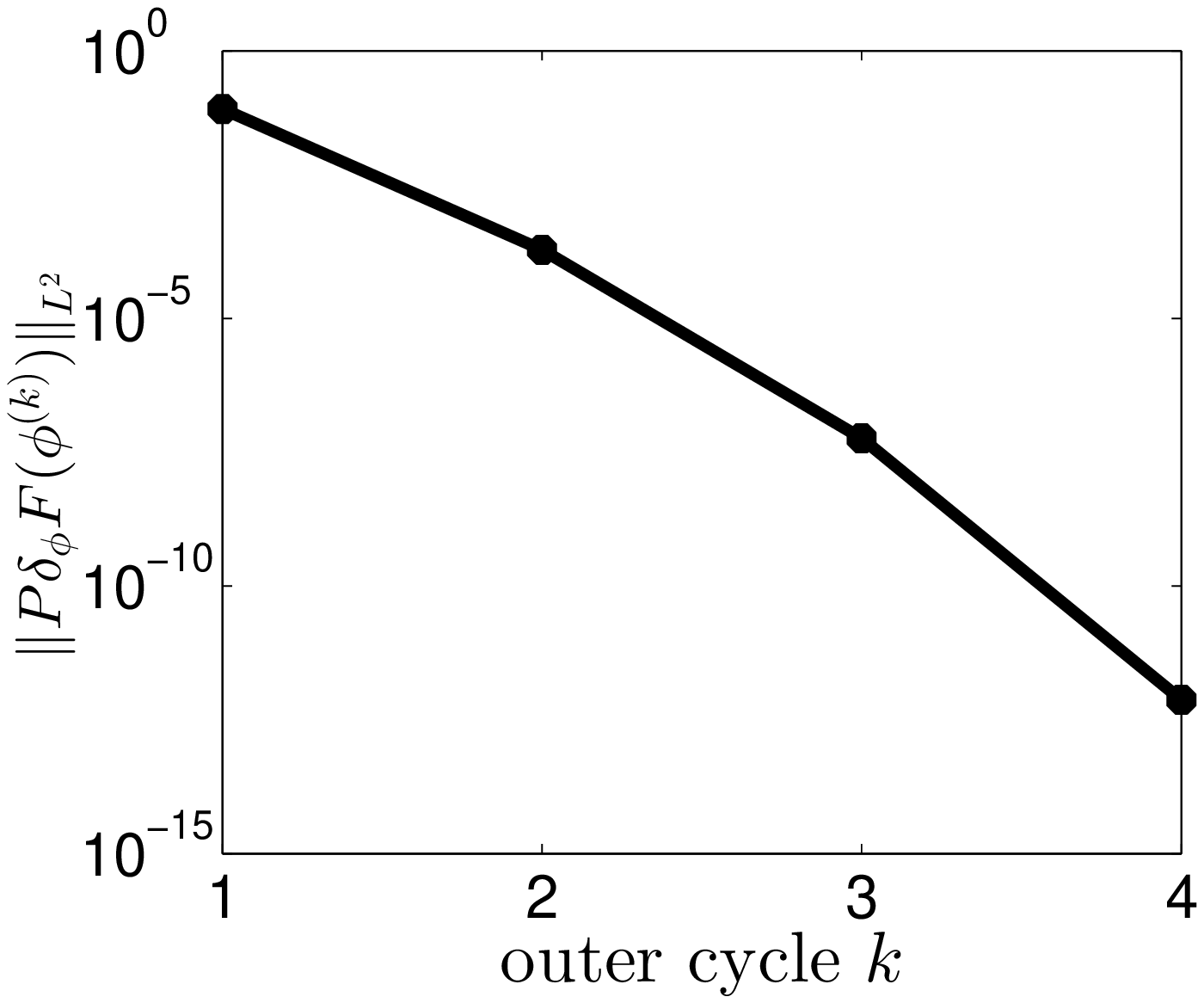}
\caption{convergence rate}
\label{figB:ProjIMF_convergence_rate}
\end{subfigure}
\caption{~  (A): initial state (dashed line); transition state by projected IMF or projected GAD (green line); transition state by IMF in $H^{-1}$ metric (red line).
 (B): The decay of the error $\| \Pj \delta_\phi F(\phi^{(k)})\|_{L^2}$ measured by the $L^2$ norm of the projected force at each cycle $k$.}
\label{fig:ProjIMF_results}
\end{center}
\end{figure}

\begin{table}[htbp]
\centering
\begin{tabular}{| c | c | c | c | c | c | c | c | c |  }
\hline
\multicolumn{2}{|c|}
{iterN}  & $1e4$ & $2e4$ & $5e4$ & $1e5$  & $2e5$ \\
\hline
\multirow{2}*{$\phi_{01}$}  & 
IMF & 16.11 & 32.06 & 80.52  &  158.75  &  320.99 \\
\cline{2-7} 
& Projected IMF  & 9.14  & 18.14 & 45.94  &  91.07 &  182.68 \\
\hline
\hline
\multirow{2}*{$\phi_{02}$}  & 
IMF  & 15.73 & 32.42 & 80.55  &  158.87 &  316.28 \\
\cline{2-7} 
& Projected IMF & 9.30  & 18.26 &  45.27  &  90.75  &  179.99   \\
\hline
\hline
\multirow{2}*{$\phi_{03}$}  & 
IMF  & 16.02 &  33.21  &  80.24  &  160.02 &  325.60 \\
\cline{2-7} 
& Projected IMF & 9.17  & 18.43 &  45.71  &  91.13  &  183.00   \\
\hline
\end{tabular}
\vskip 2mm
\caption{CPU time (seconds) comparison. ``IMF" means the original IMF in $H^{-1}$ metric; ``Projected IMF" is in $L^2$ metric.} \label{1D_case}
\end{table}

\subsection{2D example: Landau-Brazovskii free energy}
In this section, we study the nucleation problem of phase transition in diblock copolymers \cite{Tiejun2D, wise2009energy}, which have attracted a lot attention because of their various and abundant microstructures.  The model is described by the two-dimensional Landau-Brazovskii energy functional of the order parameter $\phi$,
\begin{equation}\label{F_LB}
F(\phi) = \int_\Omega   \frac{\xi^2}{2}[(\Delta+1)\phi(\textbf{r})]^2 + \Phi(\phi) ~d\,\textbf{r},
\end{equation}
defined on $\Omega = [0,\frac{16 \pi}{\sqrt{3}}] \times [0,8 \pi]$, where $\Phi(\phi) = \frac{\tau}{2}\phi^2 - \frac{\gamma}{3!}\phi^3 + \frac{1}{4!}\phi^4$.
The parameters are $\tau=-0.15, \xi=1.0, \gamma = 0.25$.
We hope to calculate the transition state of $F(\phi)$
 in the $H^{-1}$ metric. First we calculate the first and the second order variations as follows
 \begin{align*}
 \delta_\phi F(\phi) &= \xi^2(\Delta +1)^2 \phi(\textbf{r}) + \Phi^\prime(\phi), \\
 \delta^2_\phi F(\phi) &= \xi^2(\Delta +1)^2 + \Phi^{\prime\prime}(\phi):=\Hess,
 \end{align*}
 where $\Phi^\prime(\phi) = \tau \phi - \frac{\gamma}{2}\phi^2 + \frac{1}{3!}\phi^3, \Phi^{\prime\prime}(\phi) = \tau - \gamma \phi + \frac{1}{2}\phi^2$.
So the projected IMF is
  \begin{numcases}{}
v^{(k+1)} = \operatornamewithlimits{argmin}_{\|v\|=1} \left\langle v, \Pj\Hess\Pj(\phi^{(k)})v \right\rangle_{L^2},\label{IMF_2Dex1}\\
\phi^{(k+1)} = \operatornamewithlimits{argmin}_{\int_\Omega \phi(x) dx=m} L(\phi;\phi^{(k)},v^{(k+1)}),
\label{IMF_2Dex2}
\end{numcases}
with $L(\phi) = F(\phi) - F(\hat\phi)$, $\hat\phi$ is defined in \eqref{phi_hat}.
The second minimization  sub-problem is solved by evolving the gradient flow:
$$ \frac{\partial\phi}{\partial t}  =  - \Pj \delta_\phi L(\phi),
 $$
 where
$$  - \Pj \delta_\phi L(\phi) = -\Pj \big[ \xi^2(\Delta +1)^2 \phi + \Phi^\prime(\phi) \big] + 2  \inpd{v}{ \xi^2(\Delta +1)^2 \hat\phi + \Phi^\prime(\hat\phi) }_{L^2} \Pj v, $$
here $\phi = \phi^{(k+1)}$. And the projected GAD is

  \begin{numcases}{}
\frac{\partial \phi}{\partial t} = -\Pj \big[\xi^2(\Delta +1)^2 \phi + \Phi^\prime(\phi)\big] + 2 \frac{\inpd{\xi^2(\Delta +1)^2 \phi + \Phi^\prime(\phi)}{v}_{L^2} }{\inpd{v}{v}_{L^2}} \Pj v,\label{GAD-2DexPj1-x}\\
\gamma \frac{\partial v}{\partial t}  =  - \Pj \big[ \xi^2(\Delta +1)^2 + \Phi^{\prime\prime}(\phi) \big] \Pj v + \inpd{v} {  \xi^2(\Delta +1)^2 v + \Phi^{\prime\prime}(\phi)v }_{L^2} v.
\end{numcases}

For this two-dimensional numerical example, we consider the periodic boundary condition again. And for the convenience of saving computational cost, we apply the fast Fourier transform (FFT) for the two-dimensional case. We take the mesh points $Nx=Ny=64$ and the time step size $\Delta t = 0.1$.
The transition state can be obtained  by the projected IMF (or the projected GAD) in Figure \ref{figA:saddle_LB}. The quadratic convergence rate can also be obtained for the projected IMF shown in Figure \ref{figB:ProjIMF_conv_rate}. Similarly to the one-dimensional case, in order to illustrate the effect of the projected method, we make comparison with the original IMF in  $H^{-1}$ metric. We fix various inner iteration number for both cases and compare the required CPU time. Table \ref{2D_case} shows the CPU time comparison between the projected IMF and the original IMF in $H^{-1}$ metric with various initial states. Results show that the projected method for this example can save almost one-third computational cost.

\begin{figure}[htbp]
\begin{center}
\begin{subfigure}[b]{0.45\textwidth}
\includegraphics[width=\textwidth]{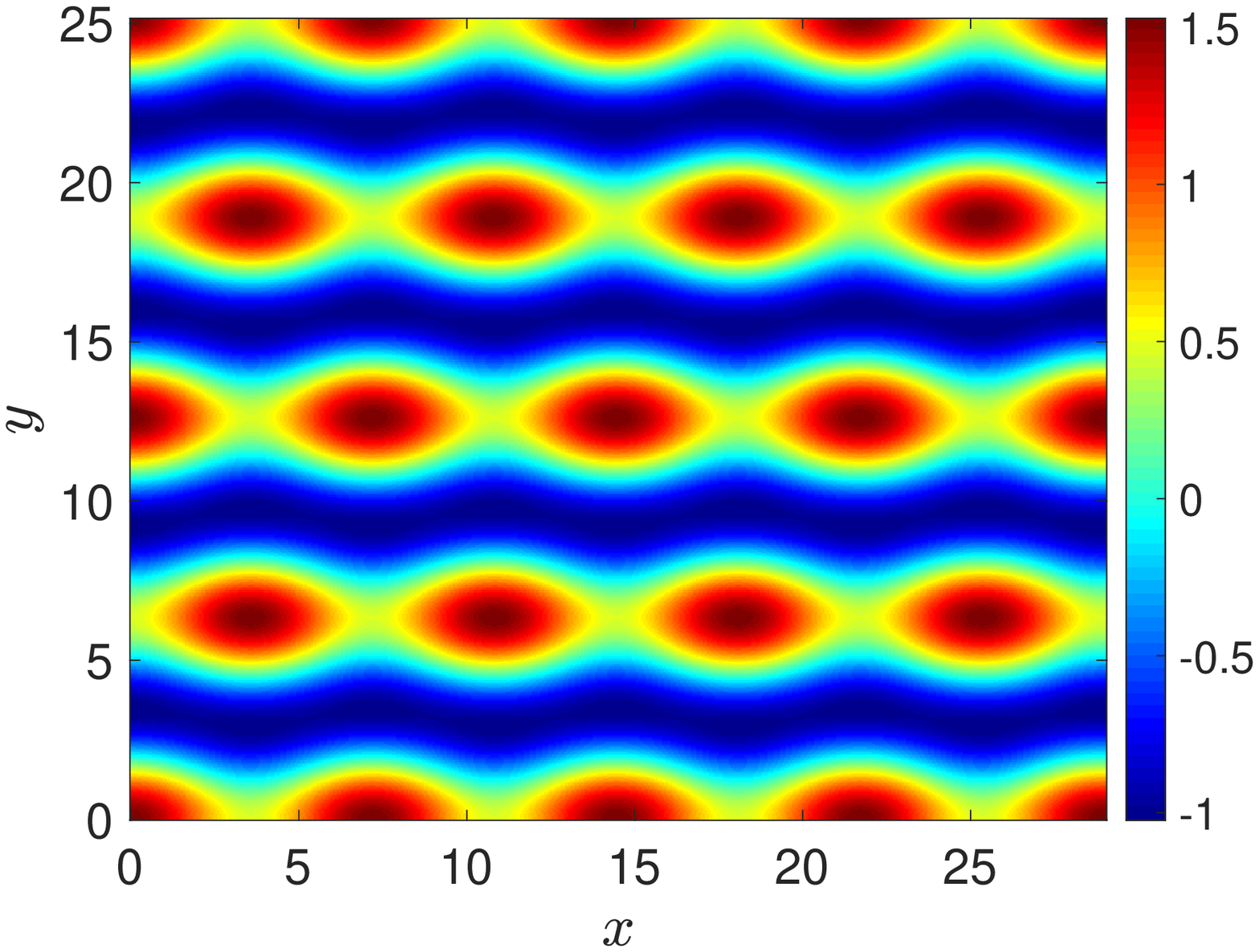}
\caption{transition state}
\label{figA:saddle_LB}
\end{subfigure}
\begin{subfigure}[b]{0.45\textwidth}
\includegraphics[width=\textwidth]{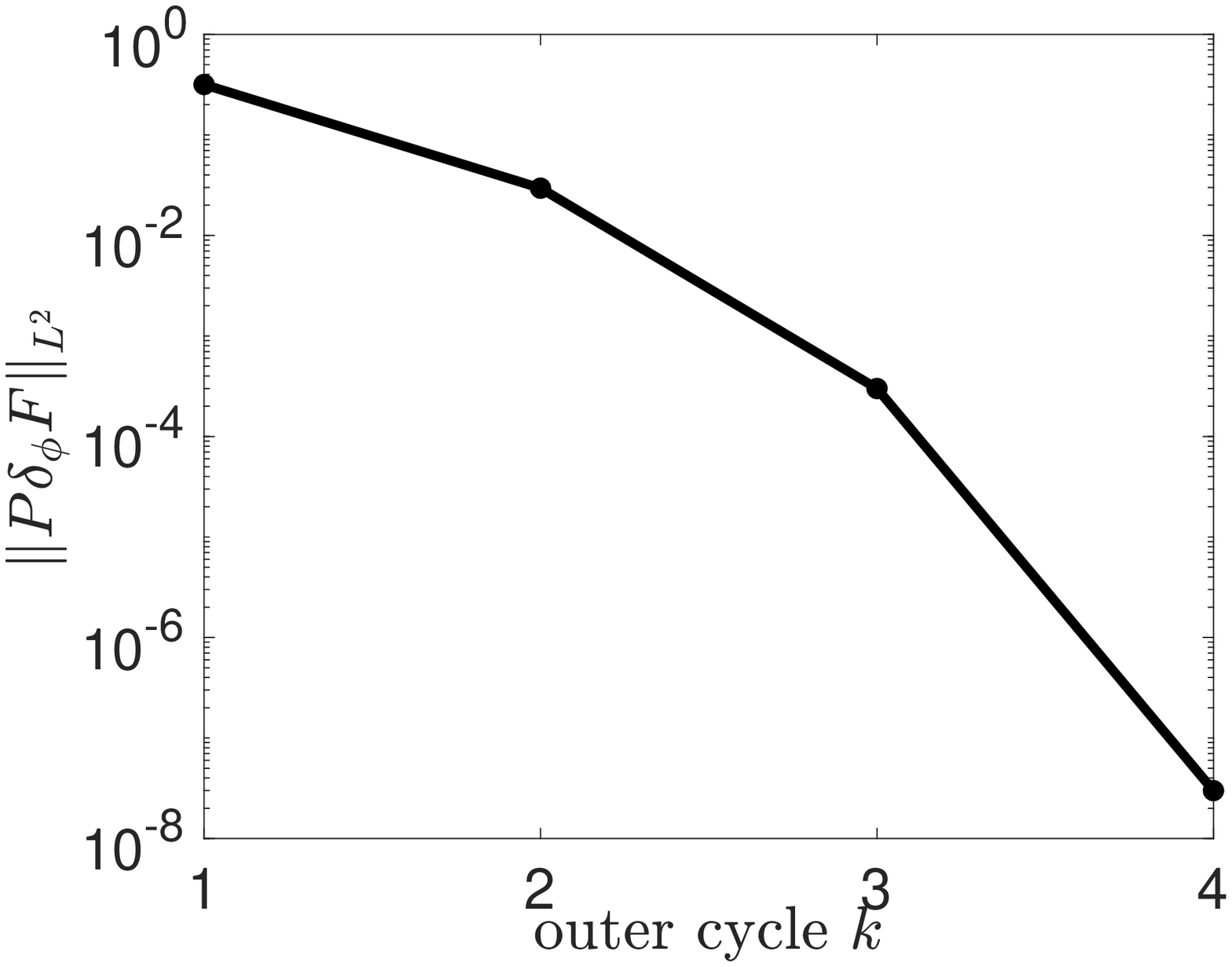}
\caption{convergence rate}
\label{figB:ProjIMF_conv_rate}
\end{subfigure}
\caption{~ (A): Transition state of the Landau Brazovskii free energy in the $H^{-1}$ metric by the projected IMF. (B): The decay of the error $\| \Pj \delta_\phi F(\phi^{(k)})\|_{L^2}$ measured by the $L^2$ norm of the projected force at each cycle $k$.}
\label{saddle_LB}
\end{center}
\end{figure}

\begin{table}[htbp]
\centering
\begin{tabular}{| c | c | c | c | c | c | c | c | c | c |  }
\hline
\multicolumn{2}{|c|}
{iterN}    &  $5e3$  & $6e3$  &  $7e3$ &  $8e3$  & $9e3$  & $1e4$   \\
\hline
\multirow{2}*{$\phi_{01}$}  & 
       IMF               &  12.62   &  15.27  &  17.67  &  20.35   &  22.71  &  25.84  \\
\cline{2-8} 
& Projected IMF  &  9.03     &   11.05  &  12.71 &  14.07  &  16.06  &  17.62  \\
\hline
\hline
\multirow{2}*{$\phi_{02}$}  & 
IMF  & 12.71  &  15.56   &  17.93   &  20.67   &  22.97    &  25.32   \\
\cline{2-8} 
& Projected IMF & 9.05  & 10.71 &  12.67   &  14.47  &  16.34   &  17.76  \\
\hline
\hline
\multirow{2}*{$\phi_{03}$}  & 
  IMF     &   12.77   &   15. 82  &    17.39   &   19.90   &   22.01   &   25.51  \\
\cline{2-8} 
& Projected IMF &  9.07 &   11.15   &  12.95    &  14.56    &    16.19   &  18.07   \\
\hline
\end{tabular}
\vskip 2mm
\caption{CPU time (seconds) comparison.  ``IMF" means the original IMF in $H^{-1}$ metric; ``Projected IMF" is in $L^2$ metric.} \label{2D_case}
\end{table}

\section{Conclusion}

In this work, we present the projected method for the IMF and the GAD to calculate the transition states of some energy functional in the $H^{-1}$ metric. By introducing an orthogonal projection operator onto the confined subspace satisfying the mass conservation, the saddle points in $H^{-1}$ metric calculation can be transformed equivalently to the saddle points in $L^2$ metric with projection.  This method can reduce much computational cost. Since it
leads to a lower order spatial derivative equation for the translation step in the IMF  compared with that in $H^{-1}$ metric directly;  more importantly, it avoids the $\Delta^{-1}$ operator calculation. The same phenomenon can be obtained for the projected GAD.
This projected method maintains the same convergence speed of the original GAD and IMF,
but the new algorithm is much faster than the direct method for $H^{-1}$ problem. 

\section*{Acknowledgement}
SG acknowledges the support of NSFC 11901211,  the youth innovative talent project of Guangdong province 2018KQNCX055 and the young teacher scientific research cultivation fund of South China Normal University 18KJ17. LL acknowledges the support of NSFC 11871486.
XZ acknowledges the support of Hong Kong RGC GRF grants 11337216 and 11305318.

\newpage

\bibliography{CVXIMF,my,gad,ms,MsGAD}

\begin{thebibliography}{10}

\bibitem{AC-EQ}
{\sc S.~M. Allen and J.~W. Cahn}, {\em A microscopic theory for antiphase
  boundary motion and its application to antiphase domain coarsening}, Acta
  Metallurgica, 27 (1979), pp.~1085 -- 1095.

\bibitem{Bates}
{\sc P.~W. Bates and F.~Chen}, {\em Spectral analysis and multidimensional
  stability of traveling waves for nonlocal allen–cahn equation}, Journal of
  Mathematical Analysis and Applications, 273 (2002), pp.~45--57.

\bibitem{Brachet}
{\sc M.~Brachet and J.~Chehab}, {\em Fast and stable schemes for phase fields
  models},  (2020).

\bibitem{CH-EQ}
{\sc J.~W. Cahn and J.~E. Hilliard}, {\em Free energy of a nonuniform system.
  i. interfacial free energy}, The Journal of Chemical Physics, 28 (1958),
  pp.~258--267.

\bibitem{Crippen1971}
{\sc G.~M. Crippen and H.~A. Scheraga}, {\em Minimization of polypeptide energy
  : {XI}. the method of gentlest ascent}, Arch. Biochem. Biophys., 144 (1971),
  pp.~462--466.

\bibitem{Fife}
{\sc F.~P.~C. Dang~H and P.~L. A}, {\em Saddle solutions of the bistable
  diffusion equation}, Ztschrift Für Angewandte Mathematik Und Physik Zamp, 43
  (1992), pp.~984--998.

\bibitem{Dawson1989}
{\sc D.~A. Dawson and J.~G\"{a}rtner}, {\em Large deviations, free energy
  functional and quasi-potential for a mean field model of interacting
  diffusions}, vol.~78, Memoirs of American Mathematical Society, 1989.

\bibitem{LZ2009}
{\sc Q.~Du and L.~Zhang}, {\em A constrained string method and its numerical
  analysis}, Commun. Math. Sci., 7 (2009), pp.~1039--1051.

\bibitem{String2002}
{\sc W.~E, W.~Ren, and E.~Vanden-Eijnden}, {\em String method for the study of
  rare events}, Phys. Rev. B, 66 (2002), p.~052301.

\bibitem{GAD2011}
{\sc W.~E and X.~Zhou}, {\em The gentlest ascent dynamics}, Nonlinearity, 24
  (2011), p.~1831.

\bibitem{IMF2014}
{\sc W.~Gao, J.~Leng, and X.~Zhou}, {\em An iterative minimization formulation
  for saddle point search}, SIAM J. Numer. Anal., 53 (2015), pp.~1786--1805.

\bibitem{IMA2015}
\leavevmode\vrule height 2pt depth -1.6pt width 23pt, {\em Iterative
  minimization algorithm for efficient calculations of transition states},
  Journal of Computational Physics, 309 (2016), pp.~69 -- 87.

\bibitem{convex_IMF}
{\sc S.~Gu and X.~Zhou}, {\em Convex splitting method for the calculation of
  transition states of energy functional}, J. Comput. Phys., 353 (2018),
  pp.~417--434.

\bibitem{Dimer1999}
{\sc G.~Henkelman and H.~J\'{o}nsson}, {\em A dimer method for finding saddle
  points on high dimensional potential surfaces using only first derivatives},
  J. Chem. Phys., 111 (1999), pp.~7010--7022.

\bibitem{NEB1998}
{\sc H.~J\`{o}nsson, G.~Mills, and K.~W. Jacobsen}, {\em Nudged elasic band
  method for finding minimum energy paths of transitions}, in Classical and
  Quantum Dynamics in Condensed Phase Simulations, B.~J. Berne, G.~Ciccotti,
  and D.~F. Coker, eds., New Jersey, 1998, LERICI, Villa Marigola,Proceedings
  of the International School of Physics, World Scientific, p.~385.

\bibitem{MAM_XZ}
{\sc T.~Li, X.~Li, and X.~Zhou}, {\em Finding transition pathways on
  manifolds}, Multiscale Model Simul., 14 (2016), pp.~173--206.

\bibitem{Tiejun2D}
{\sc T.~Li, P.~Zhang, and W.~Zhang}, {\em Nucleation rate calculation for the
  phase transition of diblock copolymers undr stochastic cahn-hilliard
  dynamics}, Multiscale Model. Simul., 11 (2013), pp.~385--409.

\bibitem{LLinProj2010}
{\sc L.~Lin, X.~Cheng, W.~E, A.~Shi, and P.~Zhang}, {\em A numerical method for
  the study of nucleation of ordered phases}, J. Comput. Phys., 229 (2010),
  pp.~1797--1809.

\bibitem{ART1998}
{\sc N.~Mousseau and G.~Barkema}, {\em Traveling through potential energy
  surfaces of disordered materials: the activation-relaxation technique}, Phys.
  Rev. E, 57 (1998), p.~2419.

\bibitem{Ren2013}
{\sc W.~Ren and E.~Vanden-Eijnden}, {\em A climbing string method for saddle
  point search}, J. Chem. Phys., 138 (2013), p.~134105.

\bibitem{ShenYang}
{\sc J.~Shen and X.~Yang}, {\em Numerical approximations of allen-cahn and
  cahn-hilliard equations}, Discrete and Continuous Dynamical Systems, 28
  (2010), pp.~1669--1691.

\bibitem{Energylanscapes}
{\sc D.~J. Wales}, {\em Energy Landscapes with Application to Clusters,
  Biomolecules and Glasses}, Cambridge University Press, 2003.

\bibitem{wise2009energy}
{\sc S.~M. Wise, C.~Wang, and J.~Lowengrub}, {\em An energy-stable and
  convergent finite-difference scheme for the phase field crystal equation},
  SIAM J. Numer. Anal., 47 (2009), pp.~2269--2288.

\bibitem{DuJCP2012}
{\sc J.~Zhang and Q.~Du}, {\em Constrained shrinking dimer dynamics for saddle
  point search with constraints}, J. Comput. Phys., 231 (2012), pp.~4745--4758.

\bibitem{Tiejun_1D}
{\sc W.~Zhang, T.~Li, and P.~Zhang}, {\em Numerical study for the nucleation of
  one-dimensional stochastic cahn-hilliard dynamics}, Comun. Math. Sci., 10
  (2012), pp.~1105--1132.

\end{thebibliography}
 
\bibliographystyle{siam} 

\end{document}